\documentclass[letterpaper,11pt,twoside,keywordsasfootnote,addressatend,noinfoline]{article}

\usepackage{imsart}
\usepackage{fullpage}
\usepackage[english]{babel}
\usepackage{amssymb}
\usepackage{amsmath}
\usepackage{amsthm}
\usepackage{epsfig}
\usepackage{subfigure}
\usepackage{mathrsfs}
\usepackage{color}

\newtheorem{theorem}{Theorem}
\newtheorem{lemma}[theorem]{Lemma}

\renewenvironment{proof}{\noindent{\bf Proof.}}{\hspace*{2mm}~$\square$}
\newenvironment{proofof}[1]{\noindent{\bf Proof of #1.}}{\hspace*{2mm}~$\square$}

\newcommand{\N}{\mathbb{N}}
\newcommand{\Z}{\mathbb{Z}}
\newcommand{\R}{\mathbb{R}}
\newcommand{\Lat}{\mathscr{L}}

\newcommand{\norm}[1]{|\!|#1|\!|}
\newcommand{\ind}{\mathbf{1}}
\newcommand{\ep}{\epsilon}
\newcommand{\n}{\hspace*{-5pt}}

\DeclareMathOperator{\poisson}{Poisson}
\DeclareMathOperator{\card}{card}

\begin{document}

\begin{frontmatter}
\title{Contact process for the spread of knowledge}
\runtitle{Contact process for the spread of knowledge}
\author{Nicolas Lanchier, Max Mercer and Hyunsik Yun}
\runauthor{Nicolas Lanchier, Max Mercer and Hyunsik Yun}
\address{School of Mathematical and Statistical Sciences \\ Arizona State University \\ Tempe, AZ 85287, USA. \\ nicolas.lanchier@asu.edu \\ mamerce1@asu.edu \\ hyun26@asu.edu}
\maketitle

\begin{abstract} \
 This paper is concerned with a natural variant of the contact process modeling the spread of knowledge on the integer lattice.
 Each site is characterized by its knowledge, measured by a real number ranging from~0~=~ignorant to~1~=~omniscient.
 Neighbors interact at rate~$\lambda$, which results in both neighbors attempting to teach each other a fraction~$\mu$ of their knowledge, and individuals die at rate one, which results in a new individual with no knowledge.
 Starting with a single omniscient site, our objective is to study whether the total amount of knowledge on the lattice converges to zero~(extinction) or remains bounded away from zero~(survival).
 The process dies out when~$\lambda \leq \lambda_c$ and/or~$\mu = 0$, where~$\lambda_c$ denotes the critical value of the contact process.
 In contrast, we prove that, for all~$\lambda > \lambda_c$, there is a unique phase transition in the direction of~$\mu$, and for all~$\mu > 0$, there is a unique phase transition in the direction of~$\lambda$.
 Our proof of survival relies on block constructions showing more generally convergence of the knowledge to infinity, while our proof of extinction relies on martingale techniques showing more generally an exponential decay of the knowledge.
\end{abstract}



\end{frontmatter}


\section{Introduction}
\label{sec:intro}
 This paper is concerned with a natural variant of Harris' contact process~\cite{harris_1974} modeling the spread of knowledge on the integer lattice.
 The contact process is the simplest invasion model based on the framework of interacting particle systems.
 The state at time~$t$ is a configuration
 $$ \eta_t : \Z^d \longrightarrow \{0, 1 \}  \quad \hbox{where} \quad \eta_t (x) = \hbox{state of site~$x$ at time~$t$}. $$
 The process can be interpreted as an epidemic model, in which case the possible states at each lattice site are~0~=~healthy and~1~=~infected, or as a population model, in which case~0~=~vacant and~1~=~occupied.
 In the context of epidemiology, infected individuals infect each of their healthy neighbors at rate~$\lambda$, and recover at rate one.
 In particular, writing~$x \sim y$ to indicate that~$x$ and~$y$ are nearest neighbors~(distance one apart), the state at site~$x$ transitions
 $$ \begin{array}{c} 0 \to 1 \ \ \hbox{at rate} \ \ \lambda \,\sum_{y \sim x} \eta_t (y) \qquad \hbox{and} \qquad 1 \to 0 \ \ \hbox{at rate} \ \ 1. \end{array} $$
 Note that the spread of the disease~(the transition~$0 \to 1$) can also be described assuming that any two neighbors~$x$ and~$y$ interact at rate~$\lambda$ and that, if one neighbor is healthy while the other neighbor is infected, the healthy neighbor becomes infected.
 Identifying the process with the set of infected sites, and starting from the configuration~$\eta_0 = \{0 \}$ that has an infected individual at the origin in an otherwise healthy population, we say that
 $$ \begin{array}{rcl}
    \hbox{the process~$\eta_t$ survives} & \hbox{when} & P (\eta_t \neq \varnothing \ \ \forall t) > 0, \vspace*{4pt} \\
    \hbox{the process~$\eta_t$ dies out} & \hbox{when} & P (\eta_t \neq \varnothing \ \ \forall t) = 0. \end{array} $$
 The main result about the contact process is the existence of a phase transition from extinction to survival.
 More precisely, there is a nondegenerate critical value
 $$ \lambda_c = \lambda_c (\Z^d) = \inf \{\lambda : P (\eta_t \neq \varnothing \ \ \forall t) > 0 \} $$
 that depends on the spatial dimension~$d$ such that the process dies out for all~$\lambda < \lambda_c$ but survives for all~$\lambda > \lambda_c$.
 Showing that the process properly rescaled in space and time dominates oriented site percolation and using a continuity argument, Bezuidenhout and Grimmett~\cite{bezuidenhout_grimmett_1990} also proved extinction at the critical value.
 See~\cite{lanchier_2024, liggett_1985, liggett_1999} for more details about the contact process. \\ \\
{\bf Model description}.
 The objective of this paper is to study a natural variant~$\xi_t$ of the contact process modeling the transmission of knowledge in a spatially structured population.
 Each lattice site is occupied by an individual who is now characterized by its knowledge, a real number ranging from~0~=~ignorant to~1~=~omniscient, so the state at time~$t$ is a configuration
 $$ \xi_t : \Z^d \longrightarrow [0, 1] \quad \hbox{where} \quad \xi_t (x) = \hbox{knowledge at site~$x$ at time~$t$}. $$
 The dynamics depends on two parameters: a rate~$\lambda$ that measures the frequency of the interactions and a fraction~$\mu$ that measures the strength of the interactions.
 More precisely, like in the basic contact process, pairs of nearest neighbors interact at rate~$\lambda$, which now results in an exchange of knowledge between the neighbors, and individuals die at rate one, which now results in a potential loss of knowledge.
 When two neighbors~$x$ and~$y$ interact, say, at time~$t$, the two neighbors attempt to teach a fraction~$\mu \in [0, 1]$ of their knowledge to the other neighbor, but a fraction~$\xi_{t-} (x)$, respectively~$\xi_{t-} (y)$, of the knowledge taught by site~$y$, respectively site~$x$, is already known by site~$x$, respectively site~$y$, so the state at sites~$x$ and~$y$ becomes
\begin{equation}
\label{eq:transitions}
\begin{array}{rcl}
\xi_t (x) & \n = \n & \xi_{t-} (x) + \mu \,\xi_{t-} (y) (1 - \xi_{t-} (x)), \vspace*{4pt} \\
\xi_t (y) & \n = \n & \xi_{t-} (y) + \mu \,\xi_{t-} (x) (1 - \xi_{t-} (y)). \end{array}
\end{equation}
 In addition, the individual at~$x$ dies at rate one, and is immediately replaced by a new individual that has the potential to learn but has no knowledge yet, so the state at~$x$ jumps to zero.
 Figure~\ref{fig:teaching} shows snapshots of the process for different values of the parameter pair~$(\lambda, \mu)$.
\begin{figure}[t]
\centering
\scalebox{0.90}{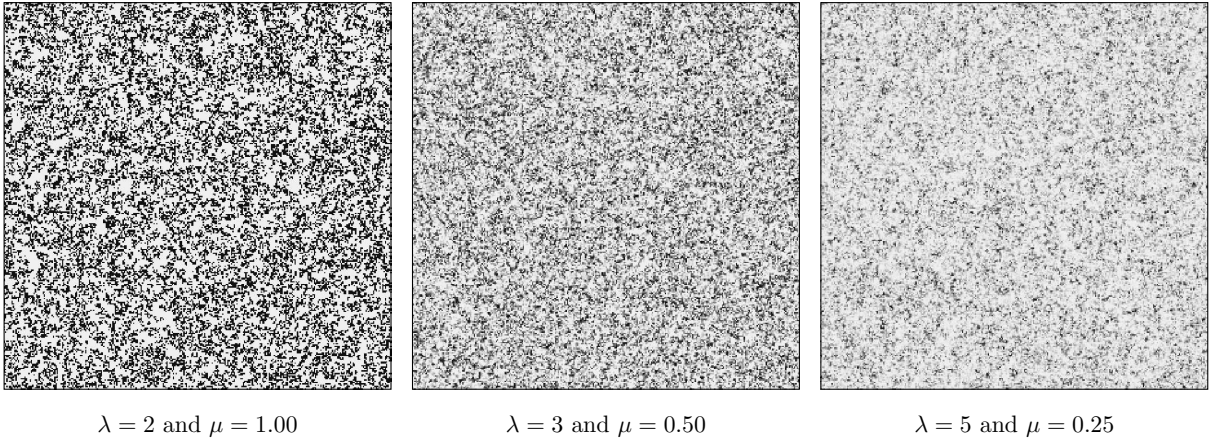}
\caption{\upshape{
 Snapshots at time~1000 of the two-dimensional process with various values of the interaction rate~$\lambda$ and the fraction~$\mu$.
 When~$\mu = 1$, the process reduces to the contact process with~0~=~ignorant and~1~=~omniscient, but the spatial distribution of knowledge becomes more uniform as~$\lambda$ increases and~$\mu$ decreases.}}
\label{fig:teaching}
\end{figure}
 Note that, at least when~$\mu > 0$, an interaction~(at rate~$\lambda$) between a site with a positive knowledge and a site with no knowledge results in the latter acquiring some knowledge, while a death~(at rate one) at a site with a positive knowledge results in the site having no more knowledge.
 This shows that the process~$\eta_t (\cdot) = \ind \{\xi_t (\cdot) > 0 \}$ that keeps track of the sites with a positive knowledge reduces to the contact process with infection rate~$\lambda$.
 However, even if~$\lambda > \lambda_c$ and so~$\eta_t$ survives, if the fraction~$\mu$ is too small, the knowledge at each infected site may decay exponentially fast and the total amount of knowledge on the lattice may converge to zero.
 Also, to define the notions of survival and extinction for the process~$\xi_t$, the relevant quantity to consider is
 $$ \begin{array}{c} \Xi_t = \sum_{x \in \Z^d} \,\xi_t (x) = \hbox{total amount of knowledge on the lattice at time~$t$}. \end{array} $$
 Then, survival and extinction are defined as follows:
\begin{equation}
\label{eq:survival-extinction}
\begin{array}{rcl}
\hbox{the process~$\xi_t$ survives} & \hbox{when} & P (\limsup_{t \to \infty} \,\Xi_t > 0) > 0, \vspace*{4pt} \\
\hbox{the process~$\xi_t$ dies out} & \hbox{when} & P (\lim_{t \to \infty} \,\Xi_t = 0) = 1. \end{array}
\end{equation}
 One of the two parameters being fixed, our objective is to study the existence/uniqueness of a phase transition from extinction to survival in the direction of the other parameter. \\ \\
{\bf Main results}.
 Important properties of the process~$\xi_t$ are monotonicity and attractiveness.
 More precisely, letting~$\xi_t^i$ for~$i = 1, 2$, be the processes with parameter pair~$(\lambda_i, \mu_i)$ and initial configuration~$\xi_0^i$, the two processes can be coupled in such a way that
\begin{equation}
\label{eq:monotone-attractive}
\begin{array}{rrl}
\lambda_1 \leq \lambda_2, & \mu_1 \leq \mu_2, & \xi_0^1 (z) \leq \xi_0^2 (z) \ \ \forall \,z \in \Z^d \vspace*{4pt} \\
                          & \Longrightarrow   & \xi_t^1 (z) \leq \xi_t^2 (z) \ \ \forall \,(z, t) \in \Z^d \times \R_+. \end{array}
\end{equation}
 In the particular case where the initial configurations are equal, the previous implication means that the process is monotone with respect to~$\lambda$ and~$\mu$, while in the particular case where the parameters are equal, it means that the process is attractive.
 It follows from~\eqref{eq:monotone-attractive} that
 $$ \begin{array}{rrrcl}
    \lambda_1 \leq \lambda_2, & \mu_1 = \mu_2, & \xi_0^1 = \xi_0^2 = \ind_0 & \Longrightarrow & (\xi_t^1 \ \hbox{survives} \ \Longrightarrow \ \xi_t^2 \ \hbox{survives}). \end{array} $$
 In particular, the parameter~$\mu$ being fixed, there is at most one phase transition from extinction to survival in the direction of the parameter~$\lambda$.
 More precisely, letting
\begin{equation}
\label{eq:critical-lambda}
\begin{array}{c} \lambda_c (\mu) = \inf \{\lambda : P (\limsup_{t \to \infty} \,\Xi_t > 0 \,| \,\xi_0 = \ind_0) > 0 \} \in [0, \infty], \end{array}
\end{equation}
 there is extinction for all~$\lambda < \lambda_c (\mu)$, and survival for all~$\lambda > \lambda_c (\mu)$.
 Similarly, it follows from the coupling that, the parameter~$\lambda$ being fixed, there is at most one phase transition from extinction to survival in the direction of the parameter~$\mu$ in the sense that, letting
\begin{equation}
\label{eq:critical-mu}
\begin{array}{c} \mu_c (\lambda) = \inf \{\mu : P (\limsup_{t \to \infty} \,\Xi_t > 0 \,| \,\xi_0 = \ind_0) > 0 \} \in [0, 1], \end{array}
\end{equation}
 there is extinction for all~$\mu < \mu_c (\lambda)$, and survival for all~$\mu > \mu_c (\lambda)$.
 One parameter being fixed, the next step is to study whether there indeed exists a phase transition in the direction of the other parameter, i.e., the critical values in~\eqref{eq:critical-lambda}--\eqref{eq:critical-mu} are nondegenerate.
 Looking first at the effects of~$\lambda$, notice that the contact process~$\eta_t$ that keeps track of the sites with a positive knowledge dominates the process~$\xi_t$, so both processes die out whenever~$\lambda \leq \lambda_c$, showing that
\begin{equation}
\label{th:lambda-small}
\lambda_c (\mu) \geq \lambda_c > 0 \quad \hbox{for all} \quad \mu \in [0, 1].
\end{equation}
 Similarly, when~$\mu = 0$, the process~$\eta_t$ reduces to the contact process with recoveries but no infections, so both processes again die out.
 However, when~$\mu > 0$ and site~$x$ has a knowledge exceeding, say, one-half, enough interactions between~$x$ and its neighbors will bring the knowledge of the neighbors above one-half.
 In particular, a block construction implies that the process survives for all~$\lambda$ large but finite.
 In terms of the critical value, we deduce that
\begin{theorem}
\label{th:lambda-large}
 For all~$\mu > 0$, we have~$\lambda_c (\mu) < \infty$.
\end{theorem}
\noindent
 In fact, the block construction used in the proof shows that, with positive probability when starting with a single omniscient site, the process converges to an invariant measure that has a positive density of sites with knowledge exceeding one-half, which leads to a somewhat stronger conclusion than survival as defined in~\eqref{eq:survival-extinction}, namely,~$\Xi_t \to \infty$ with positive probability.
 We now look at the effects of the parameter~$\mu$, starting with the subcritical phase.
 Because sites interact on average~$2d \lambda$ times with all their neighbors between consecutive deaths, it can be proved that~$\Xi_t$ is a positive supermartingale whenever~$\mu < 1 / (2d \lambda)$.
 In particular, it follows from the martingale convergence theorem that the process dies out, which gives
\begin{theorem}
\label{th:mu-small}
 For all~$\lambda < \infty$, we have~$\mu_c (\lambda) \geq 1 / (2d \lambda) > 0$.
\end{theorem}
\noindent
 Our proof shows again a stronger conclusion than extinction as defined in~\eqref{eq:survival-extinction} in the sense that we also have the following exponential decay:
 $$ E (\Xi_t^+) = \exp ((2d \lambda \mu - 1) \,t) \to 0 \quad \hbox{as} \ t \to \infty \quad \hbox{when} \quad \mu < 1 / (2d \lambda). $$
 Note also that taking~$\lambda > \lambda_c$ in the theorem gives an example where the set of sites with positive knowledge can expand without bound and still the total amount of knowledge on the lattice goes to zero.
 To conclude, we study the process with~$\lambda > \lambda_c$ and~$\mu$ close to but less than one.
 This case is more complicated.
\begin{figure}[t!]
\centering
\scalebox{0.44}{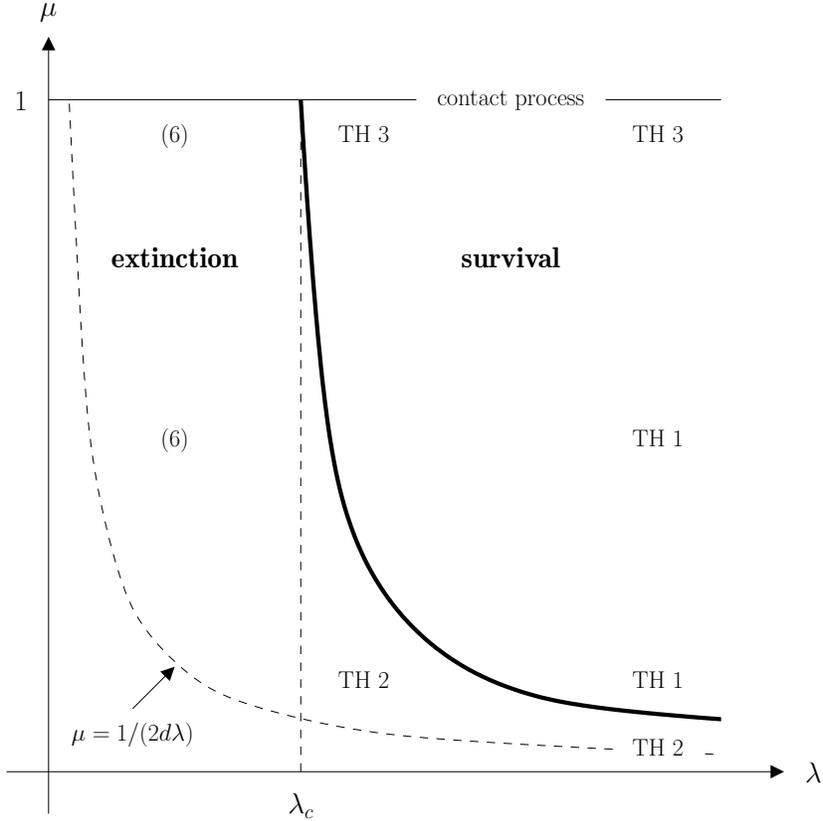}
\caption{\upshape{
 Phase structure of the process~$\xi_t$.}}
\label{fig:theorems}
\end{figure}
 When~$\mu = 1$, the process reduces to the basic contact process provided it starts with only sites in state~0 or~1, so one expects survival when~$\mu$ is close enough to one.
 To prove this result, our starting point is the block construction from~\cite{durrett_neuhauser_1997} showing that the supercritical contact process~$\eta_t$ properly rescaled in space and time dominates oriented site percolation.
 The construction shows that, with probability close to one, there are many infection paths starting at the bottom of each space-time block reaching the bottom of the~$2d$ space-time blocks immediately above.
 When~$\mu < 1$, the knowledge along these paths decays exponentially.
 However, we can prove the presence, with high probability, of double interactions along these paths that compensate for the loss of knowledge.
 More precisely, for~$\mu$ close enough to one, if these paths start with a knowledge exceeding~1/2, then the knowledge along these paths remains above~2/5 until the first double interaction, then jumps to at least~3/5, and finally remains above~1/2 until reaching the next blocks, from which we deduce survival.
 In conclusion, for all~$\lambda > \lambda_c$, the process survives for all parameters~$\mu$ sufficiently close to one, which gives
\begin{theorem}
\label{th:mu-large}
 For all~$\lambda > \lambda_c$, we have~$\mu_c (\lambda) < 1$.
\end{theorem}
\noindent
 Combining~\eqref{eq:monotone-attractive},~\eqref{th:lambda-small} and the three theorems shows that the process always dies out when~$\lambda \leq \lambda_c$ and/or~$\mu = 0$, but for all~$\lambda > \lambda_c$, there is a unique phase transition from extinction to survival in the direction of the parameter~$\mu$, while for all~$\mu > 0$, there is a unique phase transition in the direction of the parameter~$\lambda$.
 Using also that the process with~$\mu = 1$ essentially behaves like the contact process, we obtain the phase structure in Figure~\ref{fig:theorems}.
 The rest of the paper is devoted to the proof of monotonicity, attractive~\eqref{eq:monotone-attractive}, and the three theorems.

\section{Proof of~\eqref{eq:monotone-attractive} (monotonicity and attractiveness)}
\label{sec:monotonicity}
 The first step of our analysis is to use an idea of Harris~\cite{harris_1978} to construct the process graphically from collections of independent Poisson processes/exponential clocks.
 This so-called graphical representation will be used to couple processes with different parameters and/or initial configurations, and prove monotonicity and attractiveness.
 The general idea is to start with space~$\Z^d$, add one dimension for time by drawing a vertical half-line starting from each site, and visualize the dynamics using a random graph on~$\Z^d \times \R_+$.
 For our process, there are two types of transitions~(interactions occurring along the edges, and deaths occurring at each site) so we
\begin{itemize}
\item
 Place a rate~$\lambda$ exponential clock along each undirected edge~$\{x, y \}$.
 Each time the clock rings, say, at time~$t$, draw a double arrow~$(x, t) \longleftrightarrow (y, t)$, which we will call an interaction mark, and update the configuration of the process by letting
 $$ \begin{array}{rcl}
    \xi_t (x) & \n = \n & \xi_{t-} (x) + \mu \,\xi_{t-} (y) (1 - \xi_{t-} (x)), \vspace*{4pt} \\
    \xi_t (y) & \n = \n & \xi_{t-} (y) + \mu \,\xi_{t-} (x) (1 - \xi_{t-} (y)). \end{array} $$
\item
 Place a rate one exponential clock at each vertex~$x$.
 Each time the clock rings, say, at time~$t$, put a cross~$\times$ at~$(x, t)$, which we will call a death mark, and let~$\xi_t (x) = 0$.
\end{itemize}
 Harris~\cite{harris_1978} proved for a large class of interacting particle systems, including our model, that the process starting from any initial configuration can be constructed using the collection of exponential clocks and updating rules above.
 Note that the contact process~$\eta_t$ with parameter~$\lambda$ that keeps track of the sites with a positive knowledge can also be constructed using this graphical representation~(which results in a natural coupling between~$\eta_t$ and~$\xi_t$) as follows:
\begin{itemize}
\item
 If there is an interaction mark~$(x, t) \longleftrightarrow (y, t)$, and at least one of the two sites is infected, then the infection passes through the interaction mark:
 $$ (\eta_{t-} (x) = 1 \ \ \hbox{or} \ \ \eta_{t-} (y) = 1) \quad \Longrightarrow \quad (\eta_t (x) = 1 \ \ \hbox{and} \ \ \eta_t (y) = 1). $$
\item
 If there is a death mark~$\times$ at~$(x, t)$, then~$\eta_t (x) = 0$.
\end{itemize}
 We now use the graphical representation to prove the existence of the coupling in~\eqref{eq:monotone-attractive} where~$\xi_t^i$ has parameter pair~$(\lambda_i, \mu_i)$ and initial configuration~$\xi_0^i$.
 We construct~$\xi_t^1$ using the graphical representation described above with~$\lambda = \lambda_1$ and~$\mu = \mu_1$, while to construct the process~$\xi_t^2$, we also
\begin{itemize}
\item
 Place a rate~$(\lambda_2 - \lambda_1)$ exponential clock along each undirected edge~$\{x, y \}$.
 Each time the clock rings, say, at time~$t$, draw~$(x, t) \overset{\hbox{\tiny 2}}{\longleftrightarrow} (y, t)$.
\end{itemize}
 Because interaction marks~$\longleftrightarrow$ and~$\overset{\hbox{\tiny 2}}{\longleftrightarrow}$ appear at an overall rate~$\lambda_2$ by the superposition property for Poisson processes, the process~$\xi_t^2$  can be constructed as follows:
\begin{itemize}
\item
 If there is an interaction mark~$(x, t) \longleftrightarrow (y, t)$ or~$(x, t) \overset{\hbox{\tiny 2}}{\longleftrightarrow} (y, t)$, let
 $$ \begin{array}{rcl}
    \xi_t^2 (x) & \n = \n & \xi_{t-}^2 (x) + \mu_2 \,\xi_{t-}^2 (y) (1 - \xi_{t-}^2 (x)), \vspace*{4pt} \\
    \xi_t^2 (y) & \n = \n & \xi_{t-}^2 (y) + \mu_2 \,\xi_{t-}^2 (x) (1 - \xi_{t-}^2 (y)). \end{array} $$
\item
 If there is a death mark~$\times$ at~$(x, t)$, let~$\xi_t^2 (x) = 0$.
\end{itemize}
 To prove that the coupling~$(\xi_t^1, \xi_t^2)$ satisfies~\eqref{eq:monotone-attractive}, we assume that
\begin{equation}
\label{eq:assumption}
\xi_s^1 (z) \leq \xi_s^2 (z) \ \ \forall \,(z, s) \in \Z^d \times [0, t),
\end{equation}
 where~$t$ is the time at which one of the exponential clocks rings, and show that the inequality is preserved at time~$t$.
 There are three types of updates to consider.
\begin{enumerate}
\item
 Interaction mark~$(x, t) \longleftrightarrow (y, t)$. In this case,
 $$ \begin{array}{rcl}
    \xi_t^1 (x) & \n = \n &
    \xi_{t-}^1 (x) + \mu_1 \,\xi_{t-}^1 (y) (1 - \xi_{t-}^1 (x)) \vspace*{4pt} \\ & \n \leq \n &
    \xi_{t-}^1 (x) + \mu_2 \,\xi_{t-}^2 (y) (1 - \xi_{t-}^1 (x))  \vspace*{4pt} \\ & \n = \n &
    \xi_{t-}^1 (x) (1 - \mu_2 \,\xi_{t-}^2 (y)) + \mu_2 \,\xi_{t-}^2 (y) \vspace*{4pt} \\ & \n \leq \n &
    \xi_{t-}^2 (x) (1 - \mu_2 \,\xi_{t-}^2 (y)) + \mu_2 \,\xi_{t-}^2 (y)  \vspace*{4pt} \\ & \n = \n &
    \xi_{t-}^2 (x) + \mu_2 \,\xi_{t-}^2 (y) (1 - \xi_{t-}^2 (x)) = \xi_t^2 (x). \end{array} $$
 By symmetry, we also have~$\xi_t^2 (y) \leq \xi_t^2 (y)$, while
 $$ \xi_t^1 (z) = \xi_{t-}^1 (z) \leq \xi_{t-}^2 (z) = \xi_t^2 (z) \quad \hbox{for all} \quad z \neq x, y. $$
 This shows that~\eqref{eq:assumption} still holds at time~$t$. \vspace*{4pt}
\item
 Interaction mark~$(x, t) \overset{\hbox{\tiny 2}}{\longleftrightarrow} (y, t)$. In this case,
 $$ \xi_t^1 (x) = \xi_{t-}^1 (x) \leq \xi_{t-}^2 (x) \leq \xi_{t-}^2 (x) + \mu_2 \,\xi_{t-}^2 (y) (1 - \xi_{t-}^2 (x)) = \xi_t^2 (x). $$
 As before, the same inequality holds for site~$y$, while the processes remain unchanged at all the other sites, so the inequality~\eqref{eq:assumption} still holds at time~$t$. \vspace*{4pt}
\item
 Death mark~$\times$ at~$(x, t)$. In this case,~$\xi_t^1 (x) = 0 = \xi_t^2 (x)$, while the processes remain unchanged at all the other sites, so the inequality~\eqref{eq:assumption} still holds at time~$t$.
\end{enumerate}
 This proves the existence of a coupling satisfying~\eqref{eq:monotone-attractive}.


\section{Proof of Theorem~\ref{th:lambda-large}}
\label{sec:lambda-large}
 In this section, we prove that, even when the fraction~$\mu > 0$ is small, the process survives provided the interaction rate~$\lambda$ is sufficiently large.
 As previously mentioned, our proof implies in fact the following stronger result:
 Starting with~0 omniscient, there is a positive probability that the process converges to an invariant measure with a positive density of sites with knowledge exceeding, say, one-half.
 The proof relies on a block construction, a technique that first appeared in~\cite{bramson_durrett_1988} and is explained in detail in~\cite{durrett_1995}.
 The general idea is to compare the process properly rescaled in space and time with oriented site percolation.
 More precisely, let
\begin{equation}
\label{eq:sites}
\Lat = \{(z, n) \in \Z^d \times \N : z_1 + \cdots + z_d + n \ \hbox{is even} \},
\end{equation}
 that we turn into a directed graph~$\vec{\Lat}$ by placing an edge
\begin{equation}
\label{eq:arrows}
\begin{array}{rcl}
  (z, n) \to (z', n') & \Longleftrightarrow & |z_1 - z_1'| + \cdots + |z_d - z_d'| = 1 \ \hbox{and} \ n' = n + 1 \vspace*{4pt} \\
                      & \Longleftrightarrow & z' \sim z \ \hbox{and} \ n' = n + 1. \end{array}
\end{equation}
 Letting~$T > 0$ to be fixed later, and calling~$(z, n) \in \Lat$ a good site if
 $$ E_{z, n} = \{\xi_{nT} (z) > 1/2 \} = \{\hbox{the knowledge at~$z$ at time~$nT$ exceeds 1/2} \} $$
 occurs, the objective is to prove that, for all~$\ep > 0$, the time scale~$T$ can be chosen in such a way that the set of good sites dominates the set of wet sites in a percolation process on~$\vec{\Lat}$ in which sites are open with probability~$1 - \ep$.
 To begin with, we look at the process~$\xi_t$ restricted to a single edge~$\{x, y \}$ in the absence of deaths, and prove that, if the knowledge at~$x$ is at least~$1/2$, then the knowledge at both sites will be at least~$1/2$ after a finite number of interactions.
\begin{lemma}
\label{lem:interactions}
 Let~$\mu > 0$, and assume that~$\xi_0 (x) = 1/2$ and~$\xi_0 (y) = 0$.
 Then, there exists~$n = n (\mu) < \infty$ such that, after the two sites interact at least~$n$ times with no deaths,
 $$ \xi_t (x) \geq 1/2 \quad \hbox{and} \quad \xi_t (y) \geq 1/2. $$
\end{lemma}
\begin{proof}
 Let~$x_k$ and~$y_k$ be respectively the knowledge at~$x$ and at~$y$ after~$k$ interactions, and in the absence of deaths.
 It follows from the transitions~\eqref{eq:transitions} that
 $$ x_k = x_{k - 1} + \mu y_{k - 1} (1 - x_{k - 1}) \quad \hbox{and} \quad y_k = y_{k - 1} + \mu x_{k - 1} (1 - y_{k - 1}), $$
 with~$x_0 = 1/2$ and~$y_0 = 0$.
 In particular, we have~$x_k \geq x_0 = 1/2$, therefore
 $$ y_k \geq y_{k - 1} + \frac{\mu (1 - y_{k - 1})}{2} = \frac{\mu}{2} + \bigg(1 - \frac{\mu}{2} \bigg) y_{k - 1}. $$
 Finally, using a simple induction, we deduce that
 $$ \begin{array}{rcl}
    \displaystyle y_n & \n \geq \n &
    \displaystyle \frac{\mu}{2} + \bigg(1 - \frac{\mu}{2} \bigg) y_{n - 1} \geq
    \displaystyle \sum_{k = 0}^{n - 1} \bigg(\frac{\mu}{2} \bigg) \bigg(1 - \frac{\mu}{2} \bigg)^k + \bigg(1 - \frac{\mu}{2} \bigg)^n y_0 \vspace*{4pt} \\ & \n = \n &
    \displaystyle \sum_{k = 0}^{n - 1} \bigg(\frac{\mu}{2} \bigg) \bigg(1 - \frac{\mu}{2} \bigg)^k =
    \displaystyle 1 - \bigg(1 - \frac{\mu}{2} \bigg)^n. \end{array} $$
 Because the right-hand side goes to one as~$n \to \infty$ (note more precisely that it is larger than one-half whenever~$n \geq \log (1/2) / \log (1 - \mu / 2)$), the lemma follows.
\end{proof} \\ \\
 Using the previous lemma, we can now prove that, if the knowledge at~$x$ exceeds one-half then, with probability arbitrarily close to one when the interaction rate~$\lambda$ is sufficiently large, the knowledge at each of the neighbors exceeds one-half after a fixed deterministic time.
\begin{lemma}
\label{lem:invade}
 For all~$\ep, \mu > 0$, there exist~$T > 0$ and~$\lambda_+ < \infty$ such that
 $$ P (\xi_T (y) \geq 1/2 \ \ \forall \,y \sim x \,| \,\xi_0 (x) \geq 1/2) \geq 1 - \ep \quad \hbox{for all} \quad \lambda > \lambda_+. $$
\end{lemma}
\begin{proof}
 Because the process is attractive, it suffices to prove the result for the process~$\xi_t^-$ restricted to the star graph with vertex set~$\Lambda_x = \{x \} \cup \{y : y \sim x \}$.
 Let~$T > 0$ and
 $$ \begin{array}{rcl}
         D_x^T & \n = \n & \hbox{number of death marks~$\times$ in~$\Lambda_x$ by time~$T$}, \vspace*{4pt} \\
      I_{xy}^T & \n = \n & \hbox{number of interaction marks~$x \longleftrightarrow y$ by time~$T$}, \end{array} $$
 for all~$y \sim x$.
 Using that~$D_x^T = \poisson ((2d + 1) T)$, we get
\begin{equation}
\label{eq:invade-1}
 P (D_x^T = 0) = e^{- (2d + 1)T} \geq 1 - \ep/2 \quad \hbox{for some} \ T = T (\ep) > 0 \ \hbox{small}.
\end{equation}
 Then, by Lemma~\ref{lem:interactions}, there exists~$n = n (\mu) < \infty$ large such that
\begin{equation}
\label{eq:invade-2}
\xi_0^- (x) \geq 1/2 \ \ \hbox{and} \ \ D_x^T = 0 \ \ \hbox{and} \ \ I_{xy}^T \geq n \quad \Longrightarrow \quad \xi_T^- (y) \geq 1/2.
\end{equation}
 In addition, because the~$I_{xy}^T$ are independent~$\poisson (\lambda T)$, it follows from a standard Chernoff bound argument that there exists~$\lambda_+ = \lambda_+ (n, T, \ep) < \infty$ such that
\begin{equation}
\label{eq:invade-3}
\begin{array}{rcl}
\displaystyle P (I_{xy}^T \geq n \ \ \forall \,y \sim x) & \n = \n &
\displaystyle (1 - P (\poisson (\lambda T) < n))^{2d} \vspace*{8pt} \\ & \n \geq \n &
\displaystyle \bigg(1 - \bigg(\frac{e \lambda T}{n} \bigg)^n e^{- \lambda T} \bigg)^{2d} \geq 1 - \ep/2
\end{array}
\end{equation}
 for all~$\lambda > \lambda_+$.
 Combining~\eqref{eq:invade-1}--\eqref{eq:invade-3}, we deduce that, for all~$\ep, \mu > 0$, there exist~$T > 0$ and~$\lambda_+ < \infty$ such that, for all interaction rates~$\lambda > \lambda_+$,
 $$ \begin{array}{l}
      P (\xi_T (y) \geq 1/2 \ \ \forall \,y \sim x \,| \,\xi_0 (x) \geq 1/2) \vspace*{4pt} \\ \hspace*{42pt} \geq
      P (\xi_T^- (y) \geq 1/2 \ \ \forall \,y \sim x \,| \,\xi_0^- (x) \geq 1/2) \vspace*{4pt} \\ \hspace*{42pt} \geq
      P (D_x^T = 0 \ \hbox{and} \ I_{xy}^T \geq n \ \ \forall \,y \sim x) \vspace*{4pt} \\ \hspace*{42pt} =
      P (D_x^T = 0) P (I_{xy}^T \geq n \ \ \forall \,y \sim x) \geq (1 - \ep/2)^2 \geq 1 - \ep. \end{array} $$
 This completes the proof of the lemma.
\end{proof} \\ \\
 Theorem~\ref{th:lambda-large} follows from Lemma~\ref{lem:invade} and~\cite[Theorem~A.4]{durrett_1995}. \\ \\
\begin{proofof}{Theorem~\ref{th:lambda-large}}
 According to Lemma~\ref{lem:invade} and its proof that, for all~$\ep, \mu > 0$, there exists a collection of good events~$\{G_{z, n} : (z, n) \in \Lat \}$ measurable with respect to the graphical representation of the process~$\xi_t$ in the space-time region
 $$ R_{z, n} = (z, nT) + (\{-1, 0, 1 \}^d \times [0, T]) $$
 such that we can choose~$T > 0$ and~$\lambda_+ < \infty$ to guarantee that
 $$ P (G_{z, n}) \geq 1 - \ep \quad \hbox{and} \quad G_{z, n} \cap E_{z, n} \subset E_{z', n'} \ \ \forall \,(z', n') \leftarrow (z, n) $$
 for all~$\lambda > \lambda_+$.
 According to~\cite[Theorem~A.4]{durrett_1995}, this implies that the set of good sites in the interacting particle system dominates stochastically the set of wet sites in an oriented site percolation process on the directed graph~$\vec{\Lat}$ with parameter~$1 - \ep$.
 In addition, because
 $$ R_{z, n} \cap R_{z', n'} = \varnothing \quad \hbox{whenever} \quad |z_1 - z_1'| + \cdots + |z_d - z_d'| + |n - n'| \geq 3d, $$
 the percolation process has a finite range of dependence, so the parameter~$\ep > 0$ can be chosen small enough to make the percolation process supercritical.
 It follows that, with positive probability when starting with~0 omniscient, the process~$\xi_t$ converges to an invariant measure that has a positive density of sites with knowledge at least one-half.
 In particular, the process survives.
\end{proofof}


\section{Proof of Theorem~\ref{th:mu-small}}
\label{sec:mu-small}
 In this section, we prove that, for all~$\lambda < \infty$ and~$\mu < 1 / (2d \lambda)$, the process dies out.
 The proof relies on martingale techniques showing also that the expected amount of knowledge on the lattice decays exponentially fast.
 Recall that extinction is obvious when~$\leq \lambda_c$ since in this case the contact process that keeps track of the sites with a positive knowledge is subcritical.
 To deal with the supercritical case, we consider the process~$\xi_t^+ : \Z^d \longrightarrow [0, \infty)$ that again has interactions at rate~$\lambda$ and deaths at rate one, but in which an interaction between~$x$ and~$y$ results in the transitions
 $$ \xi_t (x) = \xi_{t-} (x) + \mu \,\xi_{t-} (y) \quad \hbox{and} \quad \xi_t (y) = \xi_{t-} (y) + \mu \,\xi_{t-} (x). $$
 The processes~$\xi_t$ and~$\xi_t^+$ can be coupled in such a way that~$\xi_t (z) \leq \xi_t^+ (z) \ \forall \,z$ at all times provided the inequalities hold at time~0, therefore it suffices to prove extinction of the process~$\xi_t^+$ under the assumptions of the theorem.
 The key is to control the drift of the total amount of knowledge on the lattice, which is done in the next lemma.
\begin{lemma}
\label{lem:supermartingale}
 Let~$\Xi_t^+ = \sum_{x \in \Z^d} \xi_t^+ (x)$. Then,
\begin{equation}
\label{eq:supermartingale}
\begin{array}{c} \lim_{\ep \to 0} \,\ep^{-1} \,E (\Xi_{t + \ep}^+ - \Xi_t^+ \,| \,\xi_t^+) = (2d \lambda \mu - 1) \,\Xi_t^+. \end{array}
\end{equation}
\end{lemma}
\begin{proof}
 To simplify the notation, for all~$c \neq 0$, we let
 $$ \begin{array}{c} r (c) = \lim_{\ep \to 0} \,\ep^{-1} \,P (\Xi_{t + \ep}^+ - \Xi_t^+ = c \,| \,\xi_t^+) \end{array} $$
 be the rate at which~$\Xi_t^+$ increases by~$c$ or decreases by~$- c$, and
 $$ \begin{array}{rcl}
    \Lambda_1 & \n = \n & \{x \in \Z^d : \xi_t^+ (x) > 0 \}, \vspace*{4pt} \\
    \Lambda_2 & \n = \n & \{\{x, y \} \in \Z^d \times \Z^d : x \sim y \ \hbox{and} \ \xi_t^+ (x) + \xi_t^+ (y) > 0 \}. \end{array} $$
 Because individuals die independently at rate one, which results in a loss of their knowledge, while pairs of nearest neighbors interact independently at rate~$\lambda$, which results in the transfer of a fraction~$\mu$ of their knowledge to the other neighbor, we have
 $$ \begin{array}{rclcl}
                      r (- \xi_t^+ (x)) & \n = \n & 1       & \hbox{for all} & x \in \Lambda_1, \vspace*{4pt} \\
      r (\mu (\xi_t^+ (x) + \xi_t^+ (y))) & \n = \n & \lambda & \hbox{for all} & \{x, y \} \in \Lambda_2. \end{array} $$
 Summing over all~$x \in \Lambda_1$ and all~$\{x, y \} \in \Lambda_2$, we deduce that
 $$ \begin{array}{rcl}
    \lim_{\ep \to 0} \displaystyle E \bigg(\frac{\Xi_{t + \ep}^+ - \Xi_t^+}{\ep} \ \Big| \ \xi_t^+ \bigg) & \n = \n &
    \displaystyle \sum_{x \in \Lambda_1} (- \xi_t^+ (x)) + \sum_{\{x, y \} \in \Lambda_2} \lambda \mu (\xi_t^+ (x) + \xi_t^+ (y)) \vspace*{8pt} \\ & \n = \n &
    \displaystyle - \sum_{x \in \Lambda_1} \,\xi_t^+ (x) + 2d \lambda \mu \,\sum_{x \in \Lambda_1} \,\xi_t^+ (x) = (2d \lambda \mu - 1) \,\Xi_t^+. \end{array} $$
 This completes the proof.
\end{proof} \\ \\
 Theorem~\ref{th:mu-small} follows from the previous lemma. \\ \\
\begin{proofof}{Theorem~\ref{th:mu-small}}
 Let~$\mu < 1 / (2d \lambda)$.
 Then, by Lemma~\ref{lem:supermartingale}, the process~$\Xi_t^+$ is a supermartingale with respect to the natural filtration of the process~$\xi_t^+$, and
 $$ \begin{array}{c} \sup_{t \geq 0} \norm{\Xi_t^+}_1 = \sup_{t \geq 0} E (\Xi_t^+) \leq E (\Xi_0^+) = 1 \quad \hbox{when} \quad \xi_0^+ = \ind_0. \end{array} $$
 In particular, it follows from the martingale convergence theorem that~$\Xi_t^+$ converges almost surely to a random variable~$\Xi_{\infty}^+$.
 To prove that~$\Xi_{\infty}^+ = 0$, let~$\phi (t) = E (\Xi_t^+)$, and take the expected value on both sides of equation~\eqref{eq:supermartingale} to get the differential equation
 $$ \begin{array}{rcl}
    \phi' (t) & \n = \n & \lim_{\ep \to 0} \,\ep^{-1} \,(\phi (t + \ep) - \phi (t)) = \lim_{\ep \to 0} \,\ep^{-1} \,E (\Xi_{t + \ep}^+ - \Xi_t^+) \vspace*{4pt} \\
              & \n = \n & \lim_{\ep \to 0} \,\ep^{-1} \,E (E (\Xi_{t + \ep}^+ - \Xi_t^+ \,| \,\xi_t^+)) = (2d \lambda \mu - 1) \,E (\Xi_t^+) \vspace*{4pt} \\
              & \n = \n & (2d \lambda \mu - 1) \,\phi (t). \end{array} $$
 Recalling that~$\xi_0^+ = \ind_0$, we have~$\phi (0) = \Xi_0^+ = 1$, therefore
 $$ E (\Xi_t^+) = \phi (t) = \exp ((2d \lambda \mu - 1) \,t) \to 0 \quad \hbox{as} \ t \to \infty, $$
 from which it follows that~$\Xi_{\infty}^+ = 0$.
 By stochastic domination, the same holds for~$\Xi_t$, which shows that the process dies out and completes the proof of the theorem.
\end{proofof}


\section{Proof of Theorem~\ref{th:mu-large}}
\label{sec:mu-large}
 In this section, we prove that, when the interaction rate~$\lambda > \lambda_c$ even barely, the process survives if the fraction~$\mu$ is sufficiently close to one.
 Our proof again implies the following stronger result:
 Starting with~0 omniscient, there is a positive probability that the process converges to an invariant measure with a positive density of sites with a knowledge exceeding, say, one-half.
 The result is obvious when~$\mu = 1$ since in this case the process reduces to a supercritical contact process in which infected sites are omniscient.
 To prove the result when~$\mu < 1$, we use that the contact process properly rescaled in space and time dominates oriented site percolation, and then show that the domination still holds looking instead at the sites with a knowledge exceeding one-half when~$\mu$ is close to one.
 That the contact process dominates oriented site percolation was first proved by Bezuidenhout and Grimmett~\cite{bezuidenhout_grimmett_1990}.
 However, to show our result, it will be more convenient to use the block construction from the more recent work of Durrett and Neuhauser~\cite[Section~3]{durrett_neuhauser_1997}.
 Their block construction was designed to study the multitype contact process, but ignoring one type of particle, it shows the following result for the supercritical contact process~$\eta_t$ that keeps track of the sites that carry a positive fraction of the knowledge.
 Recalling the graph~\eqref{eq:sites}--\eqref{eq:arrows}, and letting~$L$ be a large integer and~$T = L^2$, for each site~$(z, n) \in \Lat$, define the space-time blocks
 $$ \begin{array}{rcl}
      B_{z, n} & \n = \n & (Lz, nT) + ([- L, L]^d \times \{0 \}), \vspace*{4pt} \\
      R_{z, n} & \n = \n & (Lz, nT) + ([- 3L, 3L]^d \times [0, T]). \end{array} $$
 Partition the flat region~$B_{z, n}$ into small cubes of size~$L^{0.1} \times \cdots \times L^{0.1}$, and call~$(z, n) \in \Lat$ a good site if the following event occurs:
 $$ \begin{array}{l}
      E_{z, n} = \{\hbox{each of the small cubes in~$B_{z, n}$ contains} \vspace*{2pt} \\ \hspace*{70pt}
                   \hbox{at least one site with a positive knowledge} \}. \end{array} $$
 Then, for all~$\lambda > \lambda_c$ and all~$\ep > 0$, there exists a collection of good events~$\{G_{z, n} : (z, n) \in \Lat \}$ that only depend on the graphical representation in the block~$R_{z, n}$ such that
\begin{equation}
\label{eq:good-1}
  P (G_{z, n}) \geq 1 - \ep/3 \quad \hbox{and} \quad G_{z, n} \cap E_{z, n} \subset E_{z', n'} \ \ \forall \,(z', n') \leftarrow (z, n)
\end{equation}
 for all~$L$ sufficiently large.
 Our main problem is that, at each interaction, only a fraction of the knowledge is transferred, so survival of the sites with a positive knowledge does not imply that the process~$\xi_t$ survives.
 To prove survival, the objective is to improve/redesign the events~$G_{z, n}$ in order to extend the block construction to the new events
\begin{equation}
\label{eq:one-half}
\begin{array}{l}
  E_{z, n}' = \{\hbox{each of the small cubes in~$B_{z, n}$ contains} \vspace*{2pt} \\ \hspace*{70pt}
                \hbox{at least one site with a knowledge exceeding one-half} \}. \end{array}
\end{equation}
 Having a realization of the graphical representation of~$\xi_t$, we say that there is a path of length~$K$ from point~$(x, s)$ to point~$(y, t)$, and write~$(x, s) \uparrow (y, t)$, if there are
\begin{equation}
\label{eq:path}
  x = x_0, x_1, \ldots, x_K = y \quad \hbox{and} \quad s = s_0 < s_1 < \cdots < s_{K + 1} = t
\end{equation}
 such that the following two conditions hold:
\begin{itemize}
\item there is an interaction mark~$(x_{i - 1}, s_i) \longleftrightarrow (x_i, s_i)$ for~$i = 1, 2, \ldots, K$, \vspace*{4pt}
\item the segments~$\{x_{i - 1} \} \times (s_{i - 1}, s_i)$ are void of death marks~$\times$ for~$i = 1, 2, \ldots, K - 1$.
\end{itemize}
 Note that the knowledge spreads through the paths in the sense that
 $$ \xi_s (x) > 0 \quad \hbox{and} \quad (x, s) \uparrow (y, t) \quad \Longrightarrow \quad \xi_t (y) > 0. $$
 The event~$G_{z, n}$ is defined in~\cite{durrett_neuhauser_1997} as the event that, if~$(z, n)$ is good,
\begin{equation}
\label{eq:good}
\begin{array}{l}
\hbox{there is a collection of paths~$\{(x_j, nT) \uparrow (y_j, (n + 1) T) : j = 1, 2, \ldots, M \}$} \vspace*{2pt} \\
\hbox{starting from sites~$x_j \in B_{z, n}$ that are infected and reaching each of the small} \vspace*{2pt} \\
\hbox{cubes in each of the nearby regions~$B_{z', n'} \ \forall \,(z', n') \leftarrow (z, n)$}.
\end{array}
\end{equation}
 To prove that enough knowledge is transferred through these paths, we need to
\begin{itemize}
\item control their length, which measures the loss of knowledge at each jump due to interactions between different pairs of neighbors along the path, \vspace*{4pt}
\item prove the existence of multiple interactions between the same pairs of neighbors to compensate for the loss of knowledge along other parts of the path.
\end{itemize}
 The length of the paths in~\eqref{eq:good} should scale like~$T = L^2$, and so must be between~$L$ and~$L^3$ with probability close to one when~$L$ is large.
 More precisely, letting~$K_j$ denote the length of the~$j$th path in the good event~\eqref{eq:good}, and letting
\begin{equation}
\label{eq:good-2}
  H_{z, n}^- = \{K_j > L \ \ \forall \,j \leq M \} \quad \hbox{and} \quad H_{z, n}^+ = \{K_j < L^3 \ \ \forall \,j \leq M \},
\end{equation}
 we have the following result.
\begin{lemma}
\label{lem:length}
 For all~$\ep > 0$, there exists~$L_+ = L_+ (\ep) < \infty$ such that
 $$ P (H_{z, n}^-) \geq 1 - \ep/6 \quad \hbox{and} \quad P (H_{z, n}^+) \geq 1 - \ep/3 \quad \hbox{for all} \quad L > L_+. $$
\end{lemma}
\begin{proof}
 Because paths avoid the death marks, and the death marks at a given site appear at rate one, the length of the paths in~\eqref{eq:good} dominates~$X^- = \poisson (T)$.
 Now, since~$e T > L$, it follows from a standard Chernoff bound argument that
\begin{equation}
\label{eq:length-1}
  P (K_j \leq L) \leq P (X^- \leq L) \leq \bigg(\frac{eT}{L} \bigg)^L e^{- T} \leq e^{- L^2}
\end{equation}
 for all~$j \leq M$.
 Similarly, because paths jump at interaction marks, and interaction marks involving a given site appear at rate~$2d \lambda$, the length is dominated by~$X^+ = \poisson (2d \lambda T)$.
 Since in addition~$2e d \lambda T < L^3$ for all~$L$ large, it follows that
\begin{equation}
\label{eq:length-2}
  P (K_j \geq L^3) \leq P (X^+ \geq L^3) \leq \bigg(\frac{2e d \lambda T}{L^3} \bigg)^{L^3} e^{- 2d \lambda T} \leq e^{- 2d \lambda L^2}
\end{equation}
 for all~$j \leq M$.
 Note also that the number~$M$ of paths in~\eqref{eq:good} is bounded by
\begin{equation}
\label{eq:length-3}
\card \{B_{z', n'} : (z', n') \leftarrow (z, n) \} = 2d \card (B_{z, n}) = 2d (2L + 1)^d.
\end{equation}
 Combining~\eqref{eq:length-1}--\eqref{eq:length-3}, we deduce that
 $$ \begin{array}{rcl}
      P (H_{z, n}^-) & \n \geq \n & 1 - M P (X^- \leq L) \geq 1 - 2d (2L + 1)^d \,e^{- L^2}, \vspace*{4pt} \\
      P (H_{z, n}^+) & \n \geq \n & 1 - M P (X^+ \geq L^3) \geq 1 - 2d (2L + 1)^d \,e^{- 2d \lambda L^2}. \end{array} $$
 Because the right-hand sides go to one as~$L \to \infty$, the lemma follows.
\end{proof} \\ \\
 The next step is to lower bound the~\emph{overlap} of each of the paths in~\eqref{eq:good}, from which we can deduce that, with probability close to one when~$L$ is large, there is at least one~\emph{double interaction} along each of these paths.
 To define precisely the two terms in italic in the previous sentence, consider a realization of the graphical representation of the process and a path of length~$K$ characterized by the sites and times in~\eqref{eq:path}.
 For all~$i = 1, 2, \ldots, K$, let
 $$ \begin{array}{rcl}
    \sigma_i & \n = \n & \hbox{time~$s_{i - 1}$ or time of the last death mark at site~$x_i$ before time~$s_i$,} \\ &&
                         \hbox{whichever happens last}, \vspace*{4pt} \\
      \tau_i & \n = \n & \hbox{time~$s_{i + 1}$ or time of the first death mark at site~$x_{i - 1}$ after time~$s_i$,} \\ &&
                         \hbox{whichever happens first}, \end{array} $$
 as shown in Figure~\ref{fig:overlap}.
\begin{figure}[t!]
\centering
\scalebox{0.44}{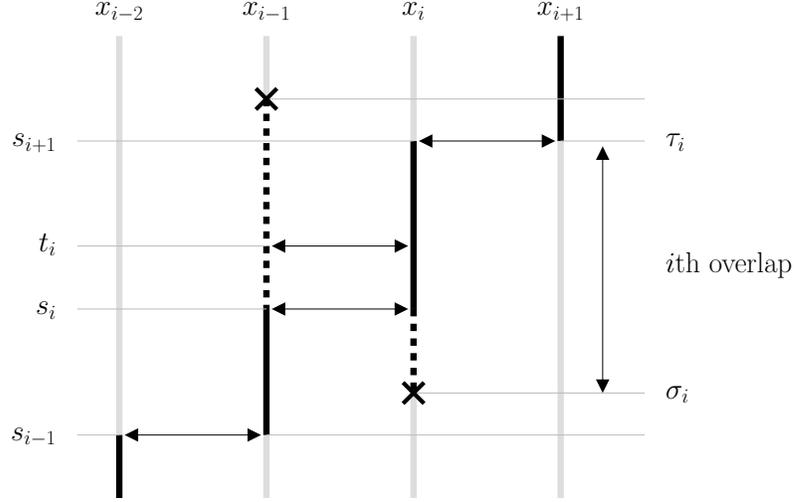}
\caption{\upshape{
 Picture of a path~(in black), of its~$i$th overlap, and of a double interaction~(at time~$t_i$).
 The dashed line at site~$x_{i - 1}$ moves upward starting at time~$s_i$ until the first death mark, and~$\tau_i = s_{i + 1}$ because the death mark appears after time~$s_{i + 1}$ moving upward.
 The dashed line at site~$x_i$ moves downward starting at time~$s_i$ until the first death mark, and~$\sigma_i$ is equal to the time of the death mark because it appears before time~$s_{i - 1}$ moving downward.}}
\label{fig:overlap}
\end{figure}
 We call the time interval~$(\sigma_i, \tau_i)$ the time window of the~$i$th overlap in the path, and define the overlap of the path as the sum
 $$ \hbox{overlap of the path} = (\tau_1 - \sigma_1) + (\tau_2 - \sigma_2) + \cdots + (\tau_K - \sigma_K). $$
 Note that, during the time window of the~$i$th overlap, the two individuals who interact at time~$s_i$ are simultaneously alive, the individual at~$x_{i - 1}$ already interacted with the previous site, and the individual at~$x_i$ did not interact yet with the next site, so the two individuals can increase the knowledge carried by the path by interacting multiple times.
 Also, we say that there is a double interaction~(during the~$i$th overlap) if there is an interaction mark
 $$ (x_{i - 1}, t_i) \longleftrightarrow (x_i, t_i) \quad \hbox{for some} \quad t_i \in (\sigma_i, \tau_i) \ \hbox{with} \ t_i \neq s_i, $$
 as shown in Figure~\ref{fig:overlap}, and define the event
\begin{equation}
\label{eq:good-3}
\begin{array}{l}
  H_{z, n} = \{\hbox{there is at least one double} \vspace*{2pt} \\ \hspace*{60pt}
             \hbox{interaction along each of the paths in~\eqref{eq:good}} \}. \end{array}
\end{equation}
 The next lemma shows that~$H_{z, n}$ occurs with probability close to one.
\begin{lemma}
\label{lem:double-interaction}
 For all~$\ep > 0$, there exists~$L_+ = L_+ (\ep) < \infty$ such that
 $$ P (H_{z, n}) \geq 1 - \ep/3 \quad \hbox{for all} \quad L > L_+. $$
\end{lemma}
\begin{proof}
 For~$j = 1, 2, \ldots, M$, and~$i = 1, 2, \ldots, K_j$, let
 $$ \begin{array}{l}
      A_{i, j} = \{\hbox{there is a double interaction} \vspace*{2pt} \\ \hspace*{60pt}
                 \hbox{during the~$i$th overlap of the~$j$th path in~\eqref{eq:good}} \}. \end{array} $$
 Having two neighbors~$x \sim y$, because interaction marks~$x \longleftrightarrow y$ occur at rate~$\lambda$, while death marks at~$x$ and interaction marks involving~$y$ occur altogether at rate~$2d \lambda + 1$, it follows from the superposition property for Poisson processes that
 $$ P (A_{i, j}) \geq \lambda / (2d \lambda + 1). $$
 In particular, conditioning on~$H_{z, n}^-$ and using Lemma~\ref{lem:length}, we get
 $$ \begin{array}{rcl}
      P (H_{z, n}^c) & \n \leq \n &
      P (H_{z, n}^c \,| \, H_{z, n}^-) + (1 - P (H_{z, n}^-)) \vspace*{4pt} \\ & \n = \n &
      P (\exists \,j \leq M : A_{i, j}^c \ \forall \,i \leq L \,| \, H_{z, n}^-) + \ep/6 \vspace*{4pt} \\ & \n \leq \n &
      M (1 - \lambda / (2d \lambda + 1))^L + \ep/6 \vspace*{4pt} \\ & \n \leq \n &
      2d (2L + 1)^d \,(1 - \lambda / (2d \lambda + 1))^L + \ep/6. \end{array} $$
 This can be made~$\leq \ep/3$ by choosing~$L$ large so the lemma follows.
\end{proof} \\ \\
 Using the block construction~\cite{durrett_neuhauser_1997} and Lemmas~\ref{lem:length} and~\ref{lem:double-interaction}, we can now prove the theorem. \\ \\
\begin{proofof}{Theorem~\ref{th:mu-large}}
 The idea is to design a collection of good events~$\{G_{z, n}' : (z, n) \in \Lat \}$ that only depend on the graphical representation in~$R_{z, n}$ such that
\begin{equation}
\label{eq:good-4}
  P (G_{z, n}') \geq 1 - \ep \quad \hbox{and} \quad G_{z, n}' \cap E_{z, n}' \subset E_{z', n'}' \ \ \forall \,(z', n') \leftarrow (z, n)
\end{equation}
 for all~$L$ large and~$\mu < 1$ close enough to one, where the events~$E_{z, n}'$ are defined in~\eqref{eq:one-half}.
 By Lemmas~\ref{lem:length} and~\ref{lem:double-interaction}, the two statements in~\eqref{eq:good-4} hold for the good events
 $$ G_{z, n}' = G_{z, n} \cap H_{z, n}^+ \cap H_{z, n} \quad \hbox{for all} \quad (z, n) \in \Lat. $$
 Indeed, according to~\eqref{eq:good-1} and the two lemmas, for all~$\ep > 0$,
 $$ P (G_{z, n}') = P (G_{z, n} \cap H_{z, n}^+ \cap H_{z, n}) \geq 1 - \ep/3 - \ep/3 - \ep/3 = 1 -\ep $$
 for all~$L$ larger than some~$L_+ = L_+ (\ep) < \infty$, which shows the first statement.
 The second statement is more difficult to prove.
 The scale parameter~$L$ being fixed, let~$\mu < 1$ such that
\begin{equation}
\label{eq:mu}
\mu^{L^3} \geq \frac{5}{6} \quad \hbox{and} \quad \frac{2 \mu}{5} \bigg(2 - \frac{2 \mu}{5} \bigg) \geq \frac{3}{5}.
\end{equation}
 The goal is to prove that, for each path~$(x, nT) \uparrow (y, (n + 1) T)$ in~\eqref{eq:good},
\begin{equation}
\label{eq:good-5}
  G_{z, n}' \ \ \hbox{and} \ \ \xi_{nT} (x) \geq 1/2 \quad \Longrightarrow \quad \xi_{(n + 1) T} (y) \geq 1/2.
\end{equation}
 Using the same notation as in~\eqref{eq:path} and Figure~\ref{fig:overlap} to define the path, because the process is attractive, it suffices to prove~\eqref{eq:good-5} under the assumption that site~$x_i$ is ignorant right before it interacts with site~$x_{i - 1}$ at time~$s_i$.
 On the good event~$G_{z, n}'$, the path has length~$K < L^3$ and contains at least one double interaction.
 Assume that the first double interaction occurs during the~$i$th overlap at times~$s_i$ and at time~$t_i$ for some~$i < L^3$.
 To fix the idea, assume also without loss of generality that~$s_i < t_i$ like in Figure~\ref{fig:overlap}.
 Then, along the path,
\begin{enumerate}
 \item the knowledge remains above~2/5 from time~$nT$ to time~$s_i$, \vspace*{4pt}
 \item the knowledge increases at time~$t_i$ and jumps to at least~3/5, \vspace*{4pt}
 \item the knowledge remains above~1/2 from time~$t_i$ to time~$(n + 1) T$.
\end{enumerate}
 Indeed, using the first inequality in~\eqref{eq:mu}, we get
 $$ \begin{array}{rcl}
    \xi_{s_i} (x_i) & \n \geq \n & \mu \xi_{s_{i - 1}} (x_{i - 1}) \geq \mu^2 \xi_{s_{i - 2}} (x_{i - 2}) \geq \cdots \vspace*{4pt} \\
                    & \n \geq \n & \mu^i \xi_{s_0} (x_0) = \mu^i \xi_{nT} (x) \geq \mu^{L^3} \xi_{nT} (x) \geq \mu^{L^3} / 2 \geq 5/12 > 2/5, \end{array} $$
 which proves the first item.
 Using~\eqref{eq:mu} and~$\xi_{s_i} (x_i) \geq \mu \xi_{s_i} (x_{i - 1}) \geq 2 \mu / 5$, we get
 $$ \begin{array}{rcl}
    \xi_{t_i} (x_i) & \n \geq \n & \xi_{s_i} (x_i) + \mu \xi_{s_i} (x_{i - 1}) (1 - \xi_{s_i} (x_i)) \vspace*{4pt} \\
                       & \n = \n & \mu \xi_{s_i} (x_{i - 1}) + \xi_{s_i} (x_i) (1 - \mu \xi_{s_i} (x_{i - 1})) \vspace*{4pt} \\
                    & \n \geq \n & \mu \xi_{s_i} (x_{i - 1}) + \mu \xi_{s_i} (x_{i - 1}) (1 - \mu \xi_{s_i} (x_{i - 1})) \vspace*{4pt} \\
                       & \n = \n & \mu \xi_{s_i} (x_{i - 1}) (2 - \mu \xi_{s_i} (x_{i - 1})) \geq (2 \mu / 5)(2 - (2 \mu / 5)) \geq 3/5, \end{array} $$
 which proves the second item.
 Finally, using again~\eqref{eq:mu}, we get
 $$ \begin{array}{rcl}
    \xi_{(n + 1) T} (y) & \n \geq \n & \xi_{s_K} (x_k) \geq \mu \xi_{s_{K - 1}} (x_{K - 1}) \geq \mu^2 \xi_{s_{K - 2}} (x_{K - 2}) \geq \cdots \vspace*{4pt} \\
                        & \n \geq \n & \mu^{K - i} \xi_{t_i} (x_i) \geq \mu^{L^3} \xi_{t_i} (x_i) \geq 3 \mu^{L^3} / 5 \geq 15/30 = 1/2, \end{array} $$
 which proves the third item, and so~\eqref{eq:good-5} and the second statement in~\eqref{eq:good-4}.
 As for Theorem~\ref{th:lambda-large}, it follows from~\eqref{eq:good-4} and~\cite[Theorem~A.4]{durrett_1995} that the set of good sites dominates the set of wet sites in an oriented site percolation process with parameter~$1 - \ep$ and a finite range of dependence, which implies that, with positive probability, the process~$\xi_t$ converges to an invariant measure that has a positive density of sites with knowledge at least one-half.
\end{proofof}



\begin{thebibliography}{100}

\bibitem{bezuidenhout_grimmett_1990}
 C. Bezuidenhout and G. Grimmett. The critical contact process dies out. \emph{Ann. Probab.}, 18(4):1462--1482, 1990.

\bibitem{bramson_durrett_1988}
 M. Bramson and R. Durrett. A simple proof of the stability criterion of Gray and Griffeath. \emph{Probab. Theory Related Fields}, 80(2):293--298, 1988.

\bibitem{durrett_1995}
 R. Durrett. Ten lectures on particle systems. In \emph{Lectures on probability theory (Saint-Flour, 1993)}, volume 1608 of \emph{Lecture Notes in Math.}, pages 97–201. Springer, Berlin, 1995.

\bibitem{durrett_neuhauser_1997}
 R. Durrett and C. Neuhauser. Coexistence results for some competition models. \emph{Ann. Appl. Probab.}, 7(1):10--45, 1997.

\bibitem{harris_1974}
 T. E. Harris. Contact interactions on a lattice. \emph{Ann. Probab.}, 2:969--988, 1974.

\bibitem{harris_1978}
 T. E. Harris. Additive set-valued {M}arkov processes and graphical methods. \emph{Ann. Probab.}, 6(3):355--378, 1978.


\bibitem{lanchier_2024}
 N. Lanchier. \emph{Stochastic interacting systems in life and social sciences}. De Gruyter Series in Probability and Stochastics, 2024. xviii + 468 pp.

\bibitem{liggett_1985}
 T. M. Liggett. \emph{Interacting particle systems, volume 276 of Grundlehren der mathematischen Wissenschaften [Fundamental Principles of Mathematical Sciences]}. Springer-Verlag, New York, 1985.

\bibitem{liggett_1999}
 T. M. Liggett. \emph{Stochastic interacting systems: contact, voter and exclusion processes, volume 324 of Grundlehren der mathematischen Wissenschaften [Fundamental Principles of Mathematical Sciences]}. Springer-Verlag, Berlin, 1999.
\end{thebibliography}
\end{document}